\documentclass{amsart}
\usepackage{graphicx}
\usepackage{amssymb}
\usepackage{epstopdf}
\usepackage{amscd}
\usepackage{amsthm}
\usepackage[all, knot]{xy}
\usepackage{extpfeil}
\xyoption{all}
\xyoption{arc}
\usepackage[hidelinks]{hyperref}
\usepackage{amsmath}
\def\cA{\mathcal A}

\def\cC{\mathcal C}

\def\cF{\mathcal F}

\def\cI{\mathcal I}

\def\cM{\mathcal M}

\def\cP{\mathcal P}

\def\a{\frak{a}}

\def\m{\frak{m}}

\def\p{\frak{p}}
\def\q{\frak{q}}

\def\N{\bN}

\newcommand{\bN}{\mathbb{N}}

\newcommand{\Z}{\mathbb{Z}}

\def\LLambda{\operatorname{{\textbf L}\Lambda}}

\def\im{\operatorname{im}}

\def\coker{\operatorname{coker}}
\def\ker{\operatorname{ker}}
\def\id{\operatorname{id}}

\def\pd{\operatorname{pd}}

\def\id{\operatorname{id}}

\def\invlim{\varprojlim}

\def\del{\partial}

\def\Spec{\operatorname{Spec}}

\def\Hom{\operatorname{Hom}}

\def\Ext{\operatorname{Ext}}
\def\Tor{\operatorname{Tor}}

\def\K{\operatorname{K}}

\def\D{\operatorname{D}}
\def\H{\operatorname{H}}

\newcommand{\Flat}{\mathcal{F}}

\newcommand{\Cot}{\mathcal{C}}
\newcommand{\PurInj}{\mathcal{PI}}

\newcommand{\Mod}{\operatorname{Mod}}

\numberwithin{equation}{section}
\newcounter{intro}
\theoremstyle{plain} 
\newtheorem{thm}[equation]{Theorem}

\newtheorem{cor}[equation]{Corollary}
\newtheorem{lem}[equation]{Lemma}
\newtheorem{prop}[equation]{Proposition}

\theoremstyle{definition}

\newtheorem{example}[equation]{Example}

\theoremstyle{remark}
\newtheorem{rem}[equation]{Remark}

\title{Minimal complexes of cotorsion flat modules}
\author{Peder Thompson}
\address{Texas Tech University, Lubbock, TX 79409}
\address{{\em Current address:} NTNU, 7491 Trondheim, Norway}
\email{peder.thompson@ntnu.no}
\date{\today}                                           

\begin{document}
\maketitle
\begin{abstract}
Let $R$ be a commutative noetherian ring.  We give criteria for a complex of cotorsion flat $R$-modules to be minimal, in the sense that every self homotopy equivalence is an isomorphism. To do this, we exploit Enochs' description of the structure of cotorsion flat $R$-modules. More generally, we show that any complex built from covers in every degree (or envelopes in every degree) is minimal, as well as give a partial converse to this in the context of cotorsion pairs. As an application, we show that every $R$-module is isomorphic in the derived category over $R$ to a minimal semi-flat complex of cotorsion flat $R$-modules. 
\end{abstract}


\section*{Introduction}
One of the most ubiquitous examples of a minimal chain complex is that of a minimal free resolution, introduced by Hilbert in the 1890s. Minimal projective and injective resolutions are an integral part of homological algebra, and there are useful criteria for identifying whether a such a complex is minimal. In particular, a projective resolution $P$ of a finitely generated module over a local ring $(R,\m)$ is minimal if and only if $P\otimes_R R/\m$ has zero differential; an injective resolution $I$ is minimal if and only if $\Hom_R(R/\p,I)_\p$ has zero differential for every prime $\p$.   In \cite{AM02}, Avramov and Martsinkovsky introduced a versatile notion of minimality for chain complexes, which recovers both of these classical notions: A chain complex $C$ is {\em minimal} if every homotopy equivalence $\gamma:C\to C$ is an isomorphism. One of our goals is to give criteria, in the spirit of these classical conditions, for a chain complex of cotorsion flat modules (defined below) to be minimal. 

Let $R$ be a commutative noetherian ring. We say an $R$-module $M$ is {\em cotorsion flat} if it is both flat and satisfies the added assumption that $\Ext_R^1(F,M)=0$ for every flat $R$-module $F$, i.e., $M$ is also cotorsion.  Enochs showed \cite{Eno84} that cotorsion flat $R$-modules have a unique decomposition, indexed by the primes of $R$, similar to the decomposition for injective modules given by Matlis \cite{Mat58}.  We use this description to characterize minimal complexes of cotorsion flat $R$-modules. In a subsequent paper \cite{Tho17b}, we show that minimal complexes of cotorsion flat $R$-modules are useful in computing cosupport, an invariant homologically dual to support that was introduced by Benson, Iyengar, and Krause in \cite{BIK12}. 

Parallel to the minimality criteria for complexes of projective or injective $R$-modules, one of our goals is to show (Theorem \ref{minimal_CF_complexes}) that a complex $B$ of cotorsion flat $R$-modules is minimal if and only if either of the following criteria hold for every $\p\in \Spec R$:
\begin{itemize}
\item The complex $R/\p\otimes_R \Hom_R(R_\p,\Lambda^\p B)$ has zero differential;
\item There is no subcomplex of the form $0 \to \widehat{R_\p}^\p\xrightarrow{\cong} \widehat{R_\p}^\p\to 0$ that is degreewise a direct summand of $B$.
\end{itemize}
The first condition has been studied previously in the context of flat resolutions of cotorsion modules, where the numbers $\dim_{\kappa(\p)}\Tor_i^{R_\p}(\kappa(\p),\Hom_R(R_\p,M))$ were examined by Enochs and Xu in \cite{Xu95,EX97} (here, $\kappa(\p)=R_\p/\p R_\p$). More recently, Dailey showed in \cite[Theorem 4.2.8]{Dai16} that a flat resolution $F$ of a cotorsion module is built from flat covers if and only if the complex $\kappa(\p)\otimes_R \Hom_R(R_\p,F)$ has zero differential. The notion of minimality studied by Enochs, Xu, and Dailey in this context refers to a resolution being built by flat covers, as opposed to the ``homotopic'' notion of minimality defined in \cite{AM02} and as is considered in Theorem \ref{minimal_CF_complexes}.

One of the ingredients in the proof of Theorem \ref{minimal_CF_complexes} is understanding how cotorsion flat $R$-modules can be broken down.  Roughly, we show that $\p$-adic completion $\Lambda^\p(-)$ and colocalization $\Hom_R(R_\p,-)$ allow us to focus on the ``$\p$-component'' of a cotorsion flat $R$-module; see Lemma \ref{CF_structure}.  After proving Theorem \ref{minimal_CF_complexes}, we end Section \ref{section_minimality} by remarking that completion and colocalization both preserve minimal complexes of cotorsion flat $R$-modules; see Proposition \ref{coloc_and_comp_preserve_minimal}.

We also study the relationship between minimality and covers/envelopes (whose definitions are recalled in \ref{covers_envelopes_resolutions}) more generally in Section \ref{covers_envelopes_minimality}. There we show that constructing complexes from covers (or envelopes) leads to a stronger form of minimality than the homotopic version introduced by Avramov and Martsinkovsky. We prove (Theorem \ref{cov_env_min}) that a complex of $R$-modules which is built entirely from covers or from envelopes (in a fixed class of modules) must be minimal.

Although not every minimal complex is built in this way (see Example \ref{example_min}), we do give a partial converse in the context of cotorsion pairs, see Proposition \ref{converse_min}. An immediate consequence is that projective resolutions of modules are minimal if and only if they are built from projective covers in every degree; injective resolutions of modules are minimal if and only if they are built from injective envelopes in every degree. We end Section \ref{covers_envelopes_minimality} with an application to characterizing perfect rings by the existence of minimal projective resolutions for all modules.

In Section \ref{section_minimal_CF}, we prove (Theorem \ref{CFreplacement}) the existence of minimal left and right cotorsion flat resolutions for certain modules, as well as show that every $R$-module is isomorphic in the derived category over $R$ to a minimal semi-flat complex of cotorsion flat $R$-modules.

\section{Preliminaries}\label{preliminaries2}
Throughout this paper, the ring $R$ is assumed to be commutative and noetherian.  We briefly recall some background material and set notation needed throughout. 

\subsection{Complexes}
A {\em complex} of $R$-modules (or shorter, {\em $R$-complex}) is a sequence of $R$-modules and $R$-linear maps
$$C=\cdots \xrightarrow{\del_C^{i-1}} C^i \xrightarrow{\del_C^i} C^{i+1}\xrightarrow{\del_C^{i+1}}\cdots$$
such that $\del_C^{i+1}\del_C^i=0$ for all $i\in \Z$.  For $R$-complexes $C$ and $D$, a {\em degree zero chain map} $f:C\to D$ is a collection of $R$-linear maps $\{f^i:C^i\to D^{i}\}_{i\in \Z}$, satisfying 
$f^{i+1}\del_C^i=\del_D^i f^i$. These are the morphisms in the category of $R$-complexes.  
We say that an $R$-complex $C$ is {\em bounded on the left (respectively, right)} if $C^i=0$ for $i\ll0$ (respectively, $C^i=0$ for $i\gg0$).  
As is standard, we set $C_i=C^{-i}$. For $R$-complexes $C$ and $D$, the total tensor product complex $C\otimes_R D$ is the direct sum totalization of the evident double complex, and the total Hom complex $\Hom_R(C,D)$ is the direct product totalization of the underlying double complex (see \cite{Wei94} 2.7.1 and 2.7.4, respectively).  We say an $R$-complex $C$ is {\em exact} (or {\em acyclic}) if $\H^i(C)=0$ for all $i\in \Z$.

\subsection{Homotopy and derived categories}
We say that degree zero chain maps $f,g:C\to D$ are {\em chain homotopic}, denoted by $f\sim g$, if there exists a cohomological degree $-1$ map (called a chain homotopy) $h:C\to D$ such that $f-g=\del_Dh+h\del_C$.  An $R$-complex $C$ is {\em contractible} if $1_C\sim 0_C$.  For further details on homotopies and complexes in general, see for example \cite{Avr98}. 

The {\em homotopy category} $\K(R)$ is the category whose objects are complexes of $R$-modules and morphisms are degree zero chain maps up to chain homotopy.  If we further invert all quasi-isomorphisms between $R$-complexes (degree zero chain maps that induce an isomorphism on cohomology), we obtain the {\em derived category} of $R$, denoted $\D(R)$.  We use $\simeq$ to denote isomorphisms in $\D(R)$. For more details on the derived category, see for example \cite[Chapter 10]{Wei94}.

We say that an $R$-complex $F$ is {\em semi-flat} (also called DG-flat, as in \cite{AF91}) if $F^i$ is flat for all $i\in\Z$ and the functor $F\otimes_R-$ preserves quasi-isomorphisms.  For example, any bounded on the right complex of flat $R$-modules is semi-flat \cite[Example 1.1.F]{AF91}.

\subsection{Injective modules}
Over a commutative noetherian ring $R$, we have a decomposition of injective $R$-modules, due to Matlis \cite{Mat58}. In fact, there exists a bijection between prime ideals $\p $ of $\Spec R$ and indecomposable injective modules $E(R/\p)$, the injective hull of $R/\p$ over $R$.  In this way, for some sets $X_\p$, every injective $R$-module can be uniquely (up to isomorphism) expressed as a sum $\bigoplus_{\p \in \Spec R}E(R/\p )^{(X_\p)}$, where for a module $N$ and set $X$, we set $N^{(X)}=\bigoplus_{X}N$.  The indecomposable injective $R$-module $E(R/\p)$ is $\p$-torsion and $\p$-local \cite[page 354]{Sha69}; a module $M$ is {\em $\p$-torsion} if for every $x\in M$, there exists $n\geq 1$ such that $\p^nx=0$ and $M$ is {\em $\p$-local} if for every $y\in R\setminus \p$, multiplication by $y$ on $M$ is an automorphism.

\subsection{Completions}\label{limits_and_completions}
For an ideal $\a\subseteq R$ and an $R$-module $M$, the {\em $\a$-adic completion of $M$} is
$$\Lambda^\a M=\invlim_n (M/\a^n M),$$
or equivalently, $\Lambda^\a M=\invlim_n(R/\a^n\otimes_R M)$. We also occasionally write $\widehat{M}^\a$ to denote the completion.  As $\Lambda^\a(-)$ defines an additive functor on the category of $R$-modules, it naturally extends to a functor on the homotopy category $\Lambda^\a:K(R)\to K(R)$.  We say an $R$-complex $M$ is {\em $\a$-complete} if the natural map $M\to \Lambda^\a M$ is an isomorphism. For a nice discussion of $\Lambda^\a$ and its left derived functor $\LLambda^\a$, see \cite{PSY14}. 

\subsection{Covers, envelopes, and $\cF$-resolutions}\label{covers_envelopes_resolutions}
Let $\cF$ be a class of $R$-modules closed under isomorphisms. For an $R$-module $M$, a morphism $\phi:M\to F$ with $F\in \cF$ is an {\em $\cF$-envelope} of $M$ if:
\begin{enumerate}
\item For any map $\phi':M\to F'$ with $F'\in \cF$, there exists $f:F\to F'$ such that $f\circ \phi=\phi'$, and
\item If $f:F\to F$ is an endomorphism with $f\circ \phi=\phi$, then $f$ must be an isomorphism.
\end{enumerate}
If $\phi:M\to F$ satisfies (1) but not necessarily (2), it is called an {\em $\cF$-preenvelope}. If an $\cF$-envelope exists, it is unique up to isomorphism. A class $\cF$ is {\em enveloping} (respectively, {\em preenveloping}) if every $R$-module has an $\cF$-envelope (respectively, an $\cF$-preenvelope).  If an enveloping class contains all injective $R$-modules, the envelopes will necessarily be injections.  

{\em $\cF$-(pre)covers} and {\em (pre)covering} classes are defined dually; see \cite[Chapter 5]{EJ00} for details. In particular, if the class $\cF$ contains the ring $R$, then $\cF$-covers are surjective.  

For any ring, Xu showed that the class of cotorsion modules is enveloping if and only if the class of flat modules is covering \cite[Theorem 3.4.6]{Xu96}; shortly after, Bican, El Bashir, and Enochs showed that the class of flat modules is covering \cite{BEBE01} for any ring (as was shown for a commutative noetherian ring by Xu \cite{Xu96}).  Hence the class of cotorsion modules is enveloping. More classically, Fuchs showed that in a noetherian ring, the class of pure-injective modules is enveloping \cite{Fuc67}. 

If $\cF$ is an enveloping class, an {\em enveloping $\cF$-resolution} of $M$ is an $R$-complex 
$$0\to F^0\to F^1\to \cdots$$
with each $F^i\in \cF$, constructed so that $M\to F^0$, $\coker(M\to F^0)\to F^1$, and $\coker(F^{i-1}\to F^{i})\to F^{i+1}$ for $i\geq 1$ are $\cF$-envelopes.  Dually, if $\cF$ is a covering class, a {\em covering $\cF$-resolution} of $M$ is an $R$-complex
$$\cdots \to F_1\to F_0\to 0$$
with each $F_i\in \cF$, constructed so that $F_0\to M$, $F_1\to \ker(F_0\to M)$, and $F_{i+1}\to \ker(F_i\to F_{i-1})$ for $i\geq 1$ are $\cF$-covers. 
Observe that the augmented enveloping $\cF$-resolution $0\to M \to F^0\to F^1\to  \cdots$ and the augmented covering $\cF$-resolution $ \cdots \to F_1\to F_0\to M\to 0$ need not be exact.

\begin{rem}
Our terminology of enveloping/covering $\cF$-resolutions is intentionally non-standard to avoid collision with usage of the term ``minimal.'' What we call enveloping/covering $\cF$-resolutions are referred to as {\em minimal} left/right $\cF$-resolutions in \cite[Chapter 8]{EJ00} as well as elsewhere in the literature, but we prefer for now to reserve the term ``minimal'' to mean a minimal complex. We show later, in Theorem \ref{cov_env_min}, that enveloping/covering $\cF$-resolutions of modules are in fact minimal complexes, which justifies the existing terminology. 
\end{rem}

We continue to use the un-decorated term {\em resolution} to mean an honest resolution in the sense that the augmented sequence is exact. That is, an $R$-complex $C$ is a {\em left resolution} of an $R$-module $M$ if there exists a quasi-isomorphism $C\xrightarrow{\simeq} M$ and $C^i=0$ for $i>0$; it is a {\em right resolution} if there is a quasi-isomorphism $M\xrightarrow{\simeq} C$ and $C^i=0$ for $i<0$. 
A {\em projective resolution} of a module is a left resolution $P$ such that each $P_i$ is projective; an {\em injective resolution} of a module is a right resolution $I$ such that each $I^i$ is injective.

\section{Decomposing cotorsion flat modules}\label{cotorsion_flat_module_preliminaries}

An $R$-module $C$ is called {\em cotorsion} if $\Ext_R^1(F,C)=0$ for every flat $R$-module $F$.  All injective $R$-modules, as well as all $R$-modules of the form $\Hom_R(M,E)$ for any $R$-module $M$ and injective $R$-module $E$, are cotorsion \cite[Lemma 2.1]{Eno84}. The class of flat $R$-modules and the class of cotorsion $R$-modules form what is called a {\em cotorsion pair}; in particular, if $F$ is any $R$-module such that $\Ext_R^1(F,C)=0$ for every cotorsion $R$-module $C$, then $F$ is flat \cite[Lemma 7.1.4]{EJ00}.  An $R$-module that is both cotorsion and flat will be called a {\em cotorsion flat} $R$-module.  

Enochs showed \cite[Theorem]{Eno84} that cotorsion flat $R$-modules have a unique decomposition indexed by $\Spec R$: An $R$-module $B$ is cotorsion flat if and only if $B\cong\prod_{\p\in \Spec R} \widehat{R_\p^{(X_\p)}}^\p$, for some sets $X_\p$. Moreover, this decomposition is uniquely determined (up to isomorphism) by the ranks of the free $R_\p$-modules $R_\p^{(X_\p)}$. For any set $X$ and $\p\in \Spec R$, there is an isomorphism \cite[Lemma 4.1.5]{Xu96}:
\begin{align}\label{CF_Hom}
\Hom_R(E(R/\p),E(R/\p)^{(X)})\cong \widehat{R_\p^{(X)}}^\p.
\end{align}

The following lemma is one of our key tools in understanding the structure of complexes of cotorsion flat $R$-modules. The lemma captures the idea of recovering the $\p$-component of a cotorsion flat $R$-module. Enochs employs this idea in \cite[Proof of Theorem]{Eno84} but does not use completion and colocalization as is done below; the ``in particular'' of part (2) below can be found in \cite[Proof of Lemma 8.5.25]{EJ00}, for instance.

\begin{lem}\label{CF_structure}
Let $B\cong \prod_{\q\in \Spec R}T_\q$ be a cotorsion flat $R$-module, with $T_\q=\widehat{R_\q^{(X_\q)}}^\q$ for some sets $X_\q$.  For an ideal $\a\subseteq R$, a prime ideal $\p\in \Spec R$, and a multiplicatively closed set $S$,
\begin{enumerate}
\item[(1)] $\widehat{B}^\a\cong \displaystyle\prod_{\a\subseteq \q}T_\q$, and
\item[(2)] $\Hom_R(S^{-1}R,B)\cong \displaystyle \prod_{\q\cap S=\varnothing}T_\q$; in particular, $\Hom_R(R_\p,B)\cong \displaystyle \prod_{\q\subseteq \p}T_\q$.
\end{enumerate}
Moreover, if $B$ is a complex of cotorsion flat $R$-modules, the natural maps
$$\Hom_R(S^{-1}R,B)\hookrightarrow B\quad\text{ and }\quad B\twoheadrightarrow \widehat{B}^\a$$
are degreewise split morphisms.  In particular, the complex 
$\Hom_R(R_\p,\widehat{B}^\p)$ 
can be identified with the subquotient complex 
$$\cdots \to T_\p^i\to T_\p^{i+1}\to \cdots,$$
having differential $\del_\p=\Hom_R(R_\p,\widehat{\del}^\p)$ induced from $B$.
\end{lem}

\begin{proof}
For (1), first decompose $B$ as $\prod_{\q\in \Spec R}T_\q\cong \left(\prod_{\a\subseteq \q}T_\q\right) \oplus \left(\prod_{\a\not\subseteq \q}T_\q\right)$.  For $\a\not\subseteq\q$ and $n\in \N$, $R/\a^n\otimes_R T_\q=0$ since $T_\q$ is $\q$-local; thus 
$$\invlim_n\left(R/\a^n\otimes_R \prod_{\a\not\subseteq \q}T_\q\right)\cong \invlim_n\prod_{\a\not\subseteq\q}(R/\a^n\otimes_R T_\q)=0.$$
On the other hand, for $\a\subseteq \q$, the $\q$-complete $R$-module $T_\q$ is $\a$-complete (see \cite[Exercise 8.2]{Mat89}). As products commute with inverse limits, a product of $\a$-complete modules is again $\a$-complete; (1) follows.

For (2), setting $E:=E(R/\q)$, we have:
\begin{align*}
\Hom_R(S^{-1}R,\prod_{\q\in \Spec R} T_\q)&\cong \prod_{\q\in \Spec R} \Hom_R(S^{-1}R,\Hom_R(E,E^{(X_\q)}))\text{, by (\ref{CF_Hom}),}\\
&\cong \prod_{\q\in \Spec R} \Hom_R(E\otimes_R S^{-1}R,E^{(X_\q)})\text{, by adjointness,}\\
&\cong \prod_{\q\cap S=\varnothing} \Hom_R(E,E^{(X_\q)})\text{, as $E$ is $\q$-local, $\q$-torsion,}\\
&\cong \prod_{\q\cap S=\varnothing} T_\q\text{, again applying (\ref{CF_Hom}).}
\end{align*}
The last remarks follow from the existence of natural maps $R\to S^{-1}R$ and $R\to \widehat{R}^\a$ in conjunction with (1) and (2). 
\end{proof}

\section{Minimality criteria for complexes of cotorsion flat modules}\label{section_minimality}
One of our main results is Theorem \ref{minimal_CF_complexes} below, where we present minimality criteria for complexes of cotorsion flat $R$-modules.  As above, we use the notation 
$$T_\q=\widehat{R_\q^{(X_\q)}}^\q$$
for some prime $\q$ and index set $X_\q$.  We start with two lemmas:

\begin{lem}\label{iso_for_each_p}
For any homomorphism $f:\prod_{\q\in \Spec R} T_\q\to \prod_{\q\in \Spec R} T_\q'$, define $f_\p$ to be the composite
\[\xymatrix{
f_{\p}:&T_{\p}\ar@{^(->}[r]^{} & \prod_{\q} T_{\q}\ar[r]^{f} &\prod_{\q} T_{\q}'\ar@{->>}[r]^{}& T_{\p}',
}\]
where the outer maps are the canonical ones.  If $f_\p:T_\p \to T_\p'$ is an isomorphism for all $\p\in \Spec R$, then $f$ is an isomorphism.
\end{lem}
\begin{proof}
We define a well ordering on the set $Z=\Spec R$, for any noetherian ring $R$, as is done in \cite{Eno87}.  This will allow us to avoid assuming finite Krull dimension.  Let $Z_0$ be the set of maximal ideals of $R$.  For any ordinal $\alpha>0$, define $Z_\alpha$ to be the set of maximal elements of $Z\setminus \left( \bigcup_{\beta<\alpha}Z_\beta\right)$.  If $Z\setminus \left( \bigcup_{\beta<\alpha}Z_\beta\right)\not=\varnothing$, then $Z_\alpha\not=\varnothing$, because $R$ is noetherian.  Moreover, there exists an ordinal $\kappa$ such that $Z_\alpha=\varnothing$ for $\alpha\geq \kappa$ (else we would contradict the fact that $Z$ is a set), hence $Z$ is the disjoint union $Z=\bigcup_{\alpha<\kappa} Z_\alpha$.  We may well order each $Z_\alpha$ (e.g., \cite[Proposition 1.1.7]{EJ00}).  By \cite[Exercise 9a, page 7]{EJ00}, we may use the well orderings of each $Z_\alpha$ to well order $Z$ so that if $\p\in Z_{\alpha}$ and $\q\in Z_{\beta}$ with $\alpha<\beta$, then $\p<\q$.  We may therefore index the primes in $Z=\Spec R$ by $\alpha<\lambda$ for some ordinal $\lambda$ so that if $\beta<\alpha<\lambda$, then $\q_\beta\not\subset \q_\alpha$.

With this well ordering, the map above is $f: \prod_{\alpha<\lambda} T_{\q_\alpha} \to \prod_{\alpha<\lambda} T_{\q_\alpha}'$, with the assumption that $f_{q_\alpha}$ is an isomorphism for each $\alpha<\lambda$. For a fixed $\beta<\lambda$, we have \cite[Lemma 4.1.8]{Xu96}:
\begin{align}\label{zero_maps_Tq}
\Hom_R\left(\prod_{\beta<\alpha<\lambda}T_{\q_\alpha},T_{\q_\beta}\right)=0.
\end{align}
For each $\beta<\lambda$, we may write
\begin{align}\label{direct_sum_decomposition}
\prod_{\alpha<\lambda}T_{\q_\alpha}=\left(\prod_{\alpha\leq \beta}T_{\q_\alpha}\right) \oplus \left(\prod_{\beta<\alpha<\lambda}T_{\q_\alpha}\right)\text{, and similarly for $\prod_{\alpha<\lambda}T_{\q_\alpha}'$,}
\end{align}
and so by (\ref{zero_maps_Tq}), there exists a map $f_{\leq \beta}$ making the following diagram commute, where the vertical maps are the canonical split projections:
\[\xymatrix{
\prod_{\alpha<\lambda}T_{\q_\alpha} \ar[r]^f\ar@{->>}[d] & \prod_{\alpha<\lambda} T_{\q_\alpha}'\ar@{->>}[d]\\
\prod_{\alpha\leq \beta} T_{\q_\alpha}\ar@{-->}[r]^{f_{\leq \beta}} & \prod_{\alpha\leq \beta} T_{\q_\alpha}'
}\]
Moreover, if $\beta'<\beta$ and $\pi_{\beta'\beta}:\prod_{\alpha\leq \beta}T_{\q_\alpha}\twoheadrightarrow \prod_{\alpha\leq \beta'}T_{\q_\alpha}$ (and similarly, $\pi_{\beta'\beta}':\prod_{\alpha\leq \beta}T_{\q_\alpha}'\twoheadrightarrow \prod_{\alpha\leq \beta'}T_{\q_\alpha}'$) are the canonical projections, a diagram chase shows that $\pi_{\beta'\beta}'\circ f_{\leq \beta}=f_{\leq \beta'}\pi_{\beta'\beta}$. Since $\prod_{\alpha<\lambda}T_{\q_{\alpha}}=\invlim_{\beta}\prod_{\alpha\leq \beta<\lambda}T_{\q_{\alpha}}$ (and likewise for $\prod_{\alpha<\lambda}T_{\q_\alpha}'$) \cite[Proof of Theorem 4.1]{Eno87}, and $\{f_{\leq \beta}\}_{\beta<\lambda}$ is a morphism of inverse systems (alternatively, these inverse systems satisfy the Mittag-Leffler condition \cite[Proposition 3.5.7]{Wei94}), to show $f$ is an isomorphism, it is enough to show
$$f_{\leq \beta}:\prod_{\alpha\leq \beta<\lambda}T_{\q_{\alpha}}\to \prod_{\alpha\leq \beta<\lambda}T_{\q_{\alpha}}'$$
is an isomorphism for all $\beta<\lambda$.  To do this, we apply transfinite induction (see e.g., \cite[Proposition 1.1.18]{EJ00}).

When $\beta=0$, the definition of $f_{\leq 0}$ along with the decomposition in (\ref{direct_sum_decomposition}) shows $f_{\leq 0}=f_{\q_0}$, which is an isomorphism by hypothesis.  

For an ordinal $\beta<\lambda$ such that $\beta=\epsilon+1$ for an ordinal $\epsilon$, there exists a map making the following diagram commute by (\ref{zero_maps_Tq}):  
\[\xymatrix{
T_{\q_\beta} \ar@{^(->}[r] \ar@{-->}[d]^{\exists} & \prod_{\alpha\leq \beta}T_{\q_{\alpha}} \ar@{->>}[r] \ar[d]^{f_{\leq \beta}} &\prod_{\alpha\leq \epsilon}T_{\q_\alpha}\ar[d]^{f_{\leq \epsilon}} \\
T_{q_\beta}' \ar@{^(->}[r] & \prod_{\alpha\leq \beta} T_{\q_\alpha}' \ar@{->>}[r] & \prod_{\alpha\leq \epsilon}T_{\q_\alpha}'
}\]
Appealing to the decomposition in (\ref{direct_sum_decomposition}), we see that the vertical map on the left agrees with $f_{\q_\beta}$. 
The left and right vertical maps are isomorphisms, by hypothesis and assumption, respectively. Hence $f_{\leq \beta}$ is an isomorphism in this case.  

Finally, suppose $\beta$ is a limit ordinal and that $f_{\leq \beta'}$ is an isomorphism for all $\beta'<\beta$.  Using the fact that $f_{\leq \beta}=\invlim_{\beta'<\beta} f_{\leq\beta'}$ and that $\{f_{\leq\beta'}\}_{\beta'<\beta}$ is a morphism of inverse systems where each $f_{\leq \beta'}$ is an isomorphism, we obtain that $f_{\leq \beta}$ is an isomorphism in this case as well.  It follows that $f$ is an isomorphism.
\end{proof} 

The following observation seems to be known, but for lack of a reference we spell it out here. Compare this result also with \cite[Lemma 3.5]{PSY14}.

\begin{lem}\label{acyclic_mod_p}
Let $F$ be a semi-flat $R$-complex.  If $R/\p\otimes_R F$ is acyclic, then $\Lambda^\p(F)$ is acyclic. 
\end{lem}
\begin{proof}
The canonical short exact sequence $\p\hookrightarrow R\twoheadrightarrow R/\p$ induces a short exact sequence of $R$-complexes $\p \otimes_R F\hookrightarrow R\otimes_R F\twoheadrightarrow R/\p\otimes_R F$; hence $\p\otimes_R F\to F$ is a quasi-isomorphism, by the long exact sequence in homology \cite[Theorem 1.3.1]{Wei94}. 
Moreover, for each $n\geq 1$, there is a short exact sequence $\p^{n+1} \otimes_R F\hookrightarrow \p^n \otimes_R F\twoheadrightarrow \p^n/ \p^{n+1}\otimes_R F$ (induced by the canonical maps in $\p^{n+1}\hookrightarrow \p^n\twoheadrightarrow \p^n/\p^{n+1}$). Since $R/\p\otimes_R F$ is acyclic and semi-flat over $R/\p$, \cite[Proposition 5.7]{Spa88} yields that $\p^n / \p^{n+1}\otimes_{R} F\cong \p^n/\p^{n+1}\otimes_{R/\p} R/\p \otimes_R F$ is acyclic. Thus $\p^n \otimes_R F\hookrightarrow F$ is a quasi-isomorphism for all $n\geq 1$, hence $R/\p^n\otimes_R F$ is acyclic for all $n\geq 1$.  The inverse system $\cdots \twoheadrightarrow R/\p^2\otimes_R F\twoheadrightarrow R/\p\otimes_R F$ satisfies the Mittag-Leffler condition, and it follows by \cite[Theorem 3.5.8]{Wei94} that $\invlim_n (R/\p^n\otimes_R F)=\Lambda^\p(F)$ is also acyclic.
\end{proof}

To find an appropriate minimality criterion for complexes of cotorsion flat $R$-modules, we turn to minimal complexes of injective $R$-modules for inspiration. Recall that an $R$-complex $B$ is {\em minimal} if each homotopy equivalence $\gamma:B\to B$ is an isomorphism \cite{AM02}; equivalently, if each map $\gamma:B\to B$ homotopic to $1_B$ is an isomorphism \cite[Proposition 1.7]{AM02}.  This agrees with the classical notion of a minimal free resolution $F$ of a finitely generated module in a local ring $(R,\m)$ (i.e., one that satisfies $\del_F(F)\subseteq \m F$), or that of a minimal injective resolution $I$ (where the inclusions $\ker(\del_I^i)\hookrightarrow I^i$ are essential).  Furthermore, a complex $I$ of injective $R$-modules is minimal if and only if $\Hom_R(R/\p,I)\otimes_R R_\p$ has zero differential for every $\p\in \Spec R$ (see \cite[Lemma B.1]{Kra05} and \cite[Remark 3.15]{ILL07}). Compare this with condition (2) of the following result.  

\begin{thm}\label{minimal_CF_complexes}
Let $R$ be a commutative noetherian ring and $B$ be a complex of cotorsion flat $R$-modules.  The following conditions are equivalent:
\begin{enumerate}
\item The complex $B$ is minimal;
\item For every $\p\in \Spec R$, the complex $R/\p \otimes_R \Hom_R(R_\p,\Lambda^\p B)$ has zero differential;
\item There does not exist a subcomplex of $B$ of the form 
$$\cdots \to 0 \to \widehat{R_\p}^\p\xrightarrow{\cong} \widehat{R_\p}^\p\to 0 \to \cdots$$ 
that is degreewise a direct summand of $B$, for any $\p\in \Spec R$.
\end{enumerate}
\end{thm}
\begin{proof}
$(1) \Rightarrow (3)$: This implication follows from \cite[Lemma 1.7]{AM02}.

$(3) \Rightarrow (2)$: For each $\p\in \Spec R$, the complex $\Hom_R(R_\p,\Lambda^\p B)$ can be described as $\cdots \to T_\p^i\xrightarrow{\del_\p^i} T_\p^{i+1}\to \cdots$, with differential induced from $B$; in detail, Lemma \ref{CF_structure} gives a commutative diagram for each $i\in \Z$:
\[\xymatrix{
\prod_\q T_\q^i \ar[rr]^{\del^i} \ar@{->>}[d]&& \prod_\q T_\q^{i+1} \ar@{->>}[d]\\
\prod_{\p\subseteq \q} T_\q^i \ar[rr]^{\Lambda^\p\del^i} && \prod_{\p\subseteq \q} T_\q^{i+1}\\
T_\p^i\ar[rr]^{\del_\p^i:=\Hom(R_\p,\Lambda^\p\del^i)}\ar@{^(->}[u] && T_\p^{i+1}\ar@{^(->}[u] 
}\]
where all vertical maps are degreewise split by Lemma \ref{CF_structure}. The canonical (degreewise split) inclusion and surjection $\iota_\p^i:T_\p^i\hookrightarrow \prod_\q T_\q^i$ and $\pi_\p^{i+1}:\prod_\q T_\q^{i+1}\twoheadrightarrow T_\p^{i+1}$, respectively, agree with the splittings of the vertical maps in the diagram.  Since $\Hom_R(T_\p^i,T_{\q}^{i+1})=0$ for $\p\subsetneq \q$ by \cite[Lemma 4.1.8]{Xu96}, a diagram chase shows $\pi_\p^{i+1}\del^i\iota_\p^i=\del_\p^i$.

Now, aiming for a contradiction, suppose $\overline{\del_\p^i}=R/\p\otimes_R \del_\p^i\not=0$ for some $\p\in \Spec R$ and $i\in \Z$.  Thus there exists $u\in T_\p^i\setminus (\p R_\p) T_\p^i$ such that $\pi_\p^{i+1} \del^i\iota_\p^i(u)=v\in T_\p^{i+1}\setminus (\p R_\p) T_\p^{i+1}$.  As $T_\p^i$ is an $\widehat{R_\p}^\p$-module, there is an $\widehat{R_\p}^\p$-linear map $\alpha:\widehat{R_\p}^\p\to T_\p^i$ mapping $1\mapsto u$.  As $u\not\in (\p R_\p)T_\p^i$, there exists a map $\rho:T_\p^i\twoheadrightarrow \kappa(\p)$ such that $\rho(u)\not=0$. As $\rho \alpha$ is nonzero, it follows that $\rho\alpha$ is a flat cover by \cite[Proposition 4.1.6]{Xu96}, hence \cite[Lemma 5.2.4]{Xu96} yields that $\alpha$ is a split injection. Similarly, $\del^i\iota_\p^i\alpha$ is a split injection, using instead that there exists a nonzero map $\rho':\prod_{\q}T_\q^{i+1}\twoheadrightarrow \kappa(\p)$ such that $\rho'\del^i\iota_\p^i\alpha$ is nonzero, and then applying \cite[Lemma 5.2.4]{Xu96}.

Set $A=\cdots \to 0 \to \widehat{R_\p}^\p\xrightarrow{=} \widehat{R_\p}^\p\to 0\to \cdots$, concentrated in degrees $i$ and $i+1$, and define a map from $\phi:A\to B$ by setting $\phi^i=\iota_\p^i\alpha$, $\phi^{i+1}=\del^i\iota_\p^i\alpha$, and $\phi^j=0$ in all other degrees.  Observe that $\phi$ is a degreewise split injective chain map: $\phi^i$ is a split injection because both $\alpha$ and $\iota_\p^i$ are split injections; $\phi^{i+1}$ is a split injection by construction.  This produces a subcomplex of $B$ forbidden by (3), hence $R/\p\otimes_R \del_\p^i=0$ and (2) follows. 

$(2) \Rightarrow (1)$:  Let $\gamma:B\to B$ be a morphism that is homotopic to the identity $1_B$. Recall from \cite[Lemma 1.7]{AM02} that $B$ is minimal if and only if $\gamma$ is an isomorphism. 
Set $\gamma_\p=\Hom_R(R_\p,\Lambda^\p(\gamma))$ and $\id_\p=\Hom_R(R_\p,\Lambda^\p (1_B))$.  Since the functor defined as $R/\p\otimes_R \Hom_R(R_\p,\Lambda^\p(-))$ preserves homotopy equivalences, the morphism $R/\p\otimes_R \gamma_\p$ is homotopic to $R/\p\otimes_R \id_\p$.  By hypothesis, the differential of $R/\p\otimes_R\Hom_R(R_\p,\Lambda^\p B)$ is trivial, and it follows that $R/\p\otimes_R \gamma_\p=R/\p\otimes_R\id_\p$ for each prime $\p$. 
Fix $i$ and $\p$ and set $F=\cdots \to 0 \to  T_\p^i\xrightarrow{\gamma_\p^i} T_\p^i\to 0\to \cdots$.  Evidently, $F$ is a semi-flat $R$-complex such that $R/\p\otimes_R F$ is acyclic (since $R/\p\otimes_R \gamma_\p=R/\p\otimes_R\id_\p$).  By Lemma \ref{acyclic_mod_p}, $F$ ($\cong\Lambda^\p F$) is acyclic. 
Hence $\gamma_\p^i$ is an isomorphism for each $\p\in \Spec R$.  By Lemma \ref{iso_for_each_p}, it follows that $\gamma^i:B^i\to B^i$ is an isomorphism for every $i\in \Z$, and so (1) follows.
\end{proof}

We end this section by showing that, for a multiplicatively closed set $S$ and an ideal $\a$, the functors $\Hom_R(S^{-1}R,-)$ and $\Lambda^{\a}(-)$ preserve minimal complexes of cotorsion flat $R$-modules.

\begin{prop}\label{coloc_and_comp_preserve_minimal}
Let $B$ be a minimal complex of cotorsion flat $R$-modules, $\a\subseteq R$ an ideal, and $S$ a multiplicatively closed set. Then $\Lambda^\a(B)$ and $\Hom_R(S^{-1}R,B)$ are minimal complexes of cotorsion flat $R$-modules.
\end{prop}
\begin{proof}
By Lemma \ref{CF_structure}, both $\Lambda^\a(B)$ and $\Hom_R(S^{-1}R,B)$ are complexes of cotorsion flat modules and the maps 
$$B\twoheadrightarrow \Lambda^\a(B)\quad \text{ and } \quad \Hom_R(S^{-1}R,B)\hookrightarrow B$$ are a degreewise split surjection and injection, respectively. Applying the functor 
$$R/\p\otimes_R\Hom_R(R_\p,\Lambda^\p(-))$$ 
to either of these shows the resulting complexes must have zero differential by Theorem \ref{minimal_CF_complexes}, and hence the desired complexes are minimal.
\end{proof}

\section{Covers, envelopes, and minimal complexes}\label{covers_envelopes_minimality}
Herein we consider minimality of $R$-complexes more generally. For any class of $R$-modules $\cA$ that is closed under isomorphisms, we will show that a complex built from $\cA$-covers in every degree (or $\cA$-envelopes in every degree) is minimal, as well as prove a partial converse.  The results proved here are applied in the following section, where we study the case of cotorsion flat resolutions and replacements.

\begin{thm}\label{cov_env_min}
Let $\cA$ be a class of $R$-modules closed under isomorphisms and let $A$ be an $R$-complex with each $A_i\in \cA$. Suppose at least one of the following holds:
\begin{enumerate}
\item The canonical surjection $A_i\twoheadrightarrow \coker(\del_{i+1})$ is an $\cA$-cover for all $i\in \Z$; or
\item The canonical injection $\ker(\del_i)\hookrightarrow A_i$ is an $\cA$-envelope for all $i\in \Z$.
\end{enumerate}
Then $A$ is a minimal $R$-complex.
\end{thm}
\begin{proof}
For each $i\in \Z$, set $C_i=\coker(\del_{i+1})$ and assume the canonical surjection $\pi_i:A_i\twoheadrightarrow C_i$ is an $\cA$-cover. We first address the case where $A$ is bounded on the right, so that for some $n\in \Z$, we have $A_i=0$ for $i<n$. Let $\gamma:A\to A$ be a degree zero chain map such that $\gamma\sim 1^A$. Thus there is a homotopy $\sigma$ such that $1^A_i-\gamma_i=\del_{i+1}\sigma_i+\sigma_{i-1}\del_i$ for each $i\in \Z$.  We first claim the following diagram commutes:
\[\xymatrix{
A_n \ar@{->>}[r]^{\pi_n} \ar[d]^{\gamma_n} & C_n\ar[d]^{1^{C}_n}\\
A_n \ar@{->>}[r]^{\pi_n} & C_n
}\]
This follows because $\del_n=0$ (by the assumption that $A_i=0$ for $i<n$) and $\pi_n\del_{n+1}=0$, so
$$\pi_n(1^{A}_n-\gamma_n)=\pi_n\del_{n+1}\sigma_n+\pi_n\sigma_{n-1}\del_n=0 \implies 1^{C}_n\pi_n-\pi_n\gamma_n=0.$$ 
Because $\pi_n$ is an $\cA$-cover, we conclude that $\gamma_n$ is an isomorphism.  

Fix $i>n$ and proceed by induction, assuming $\gamma_{j}$ is an isomorphism for all $j\leq i-1$. As $\gamma$ is a chain map, $\gamma_{i-1}$ induces a map $\ker(\del_{i-1})\to \ker(\del_{i-1})$. Since $\gamma_{i-1}$ and $\gamma_{i-2}$ are isomorphisms, the five lemma yields that $\gamma_{i-1}:\ker(\del_{i-1})\to \ker(\del_{i-1})$ is an isomorphism as well. Moreover, as $\gamma\sim 1^A$, we know that $\gamma_{i-1}$ induces an isomorphism on homology. Consider the short exact sequence $B_{i-1}\hookrightarrow Z_{i-1} \twoheadrightarrow \H_{i-1}(A)$, where $B_{i-1}=\im(\del_i)$ and $Z_{i-1}=\ker(\del_{i-1})$. The diagram below commutes because $\gamma$ is a chain map:
\[\xymatrix{
B_{i-1}\ar@{^(->}[r]\ar[d]^{\gamma_{i-1}} & Z_{i-1} \ar@{->>}[r] \ar[d]^{\gamma_{i-1}}_{\cong}& \H_{i-1}(A) \ar[d]_{\cong}\\
B_{i-1} \ar@{^(->}[r] & Z_{i-1} \ar@{->>}[r] & \H_{i-1}(A)
}\]
The five lemma forces $\gamma_{i-1}:B_{i-1}\to B_{i-1}$ to also be an isomorphism. We also have a short exact sequence $\H_{i}(A)\hookrightarrow C_{i} \twoheadrightarrow A_i/Z_i\cong B_{i-1}$, which shows $\gamma_i$ induces an isomorphism on $C_i$. Therefore we have the following commutative diagram:
\[\xymatrix{
A_i \ar@{->>}[r]^{\pi_i} \ar[d]^{\gamma_i} & C_i\ar[d]^{\overline{\gamma_i}}_{\cong}\\
A_i \ar@{->>}[r]^{\pi_i} & C_i
}\]
Since $\pi_{i}$ is an $\cA$-cover, we conclude that $\gamma_i$ is an isomorphism as well. By induction, this shows $\gamma$ is an isomorphism. Thus, when $A$ is right bounded and we assume (1) holds, $A$ is minimal.

Now, still assuming (1) holds, consider the case where $A$ is no longer assumed to be bounded on the right.  Let $\gamma:A\to A$ be such that $\gamma \sim 1^A$, so that $1^{A}_i-\gamma_i=\del_{i+1}\sigma_i+\sigma_{i-1}\del_i$ for each $i\in \Z$.  Fix $j\in \Z$ and define a map $\widetilde{\gamma}:A_{\geq j-1}\to A_{\geq j-1}$, where $A_{\geq j-1}$ is a hard truncation, as:
$$\widetilde{\gamma}=
\begin{cases} \gamma_i, & i\geq j\\
\gamma_{j-1}+\sigma_{j-2}\del_{j-1}, & i=j-1\\
0, & i<j-1\end{cases}$$
Since $\sigma_{j-2}\del_{j-1}\del_j=0$, the following diagram commutes and hence $\widetilde{\gamma}$ defines a chain map.
\[\xymatrix{
A_{\geq j-1}: \ar[d]^{\widetilde{\gamma}}&\cdots \ar[r]^{\del_{j+1}} & A_j \ar[rr]^{\del_{j}} \ar[d]^{\gamma_j} && A_{j-1} \ar[rr]\ar[d]^{\gamma_{j-1}+\sigma_{j-2}\del_{j-1}} && 0\ar[d]\\
A_{\geq j-1}: &\cdots \ar[r]^{\del_{j+1}} & A_j \ar[rr]^{\del_j} && A_{j-1} \ar[rr] && 0
}\]
Moreover, $\widetilde{\gamma}\sim 1^{A_{\geq j-1}}$ where the homotopy is just given by $\sigma_{\geq j-1}=\begin{cases} \sigma_i, & i\geq j-1\\ 0, & i<j-1\end{cases}$. Since $A_{\geq j-1}$ is bounded on the right, the work above shows that $\widetilde{\gamma}$ is an isomorphism, hence $\gamma_j$ is an isomorphism.  As $j\in \Z$ was arbitrary, this yields that $\gamma$ is an isomorphism and hence $A$ is minimal.

For case (2), the argument is dual, and is left to the reader.  First consider the case where $A$ is bounded on the left and proceed inductively; for an arbitrary $R$-complex $A$, hard truncate on the left and define $\widetilde{\gamma}$ by using the homotopy to modify the leftmost nonzero map as was done above.
\end{proof}

We aim to prove a partial converse to this result, noting that the converse cannot hold in general since there are minimal $R$-complexes not built entirely from covers or entirely from envelopes:

\begin{example}\label{example_min}
Let $R$ be a local ring, $M$ a finitely generated $R$-module with $\pd_RM=1$, and $N$ an $R$-module with $\id_RN=1$.  Let $P=0\to P_1\xrightarrow{\del_1}P_0\to 0$ be the minimal projective resolution of $M$ and $I=0\to I_0\xrightarrow{\del_0}I_{-1}\to 0$ be the minimal injective resolution of $N$. Set $A=P\oplus \Sigma^{-1} I$, where $\Sigma^{-1} I$ is the complex $I$ shifted one degree to the right.

Evidently, $A$ is a minimal complex: Any $\gamma:A\to A$ that is homotopic to $1^A$ restricts to maps on $P$ and on $I$, both homotopic to their respective identity maps. As $P$ and $I$ are both minimal complexes, these maps are isomorphisms, and hence $\gamma$ is an isomorphism.

However, $A$ is not built entirely from covers or entirely from envelopes.  Let $\cA$ be any class of modules, containing 0 and closed under isomorphisms. Since $\del_1:P_1\hookrightarrow P_0$ is an injection, the canonical inclusion $0=\ker(\del_1^A)\hookrightarrow A_1\not=0$ cannot be an $\cA$-envelope. Furthermore, since $\del_0:I_0\twoheadrightarrow I_{-1}$ is a surjection, the canonical map $0\not=A_{-2}\twoheadrightarrow \coker(\del_{-1}^A)=0$ cannot be an $\cA$-cover.
\end{example}

We can say more in the context of cotorsion pairs. Recall that, for any class of $R$-modules $\cA$, one defines the orthogonal classes 
\begin{align*}
{}^\perp \cA&=\{M\in \Mod(R) \mid \Ext_R^1(M,A)=0 \text{ for all } A\in \cA\},\\
\cA^\perp&=\{N\in \Mod(R) \mid \Ext_R^1(A,N)=0 \text{ for all } A\in \cA\}.
\end{align*}
If $\cF$ and $\cC$ are classes of $R$-modules closed under isomorphisms, we say that $(\cF,\cC)$ is a {\em cotorsion pair} if $\cF={}^\perp\cC$ and $\cF^\perp = \cC$.

\begin{prop}\label{converse_min}
Let $(\cF,\cC)$ be a cotorsion pair of $R$-modules. 
\begin{enumerate}
\item If $F$ is a minimal $R$-complex with $F_i\in \cF$ and $\im(\del_i^F)\in \cC$ for all $i\in \Z$, then the canonical surjection $F_i\twoheadrightarrow \coker(\del_{i+1}^F)$ is an $\cF$-precover for all $i\in \Z$ and an $\cF$-cover for $i\geq \sup\{j\mid \H_j(F)\not=0\}$.
\item If $C$ is a minimal $R$-complex with $C^i\in \cC$ and $\im(\del^i_C)\in \cF$ for all $i\in \Z$, then the canonical inclusion $\ker(\del^i_C)\hookrightarrow C^i$ is a $\cC$-preenvelope for all $i\in \Z$ and a $\cC$-envelope for $i\geq \sup\{j\mid \H^j(C)\not=0\}$.
\end{enumerate}
\end{prop}
\begin{proof}
Let $F$ be a minimal $R$-complex as in (1) and for $i\in \Z$, set $C_i=\coker(\del_{i+1}^F)$.

Fix $n\in \Z$. To see that the canonical surjection $\pi_n:F_n\twoheadrightarrow C_n$ is an $\cF$-precover, let $G\in \cF$ and consider the short exact sequence $\im(\del_{n+1}^F)\hookrightarrow F_n\twoheadrightarrow C_n$. Since $\Ext_R^1(G,\im(\del_{n+1}^F))=0$, the surjection $\pi_n$ induces a surjection $\Hom_R(G,F_n)\twoheadrightarrow \Hom_R(G,C_n)$. It follows that any map $G\to C_n$ factors through $F_n$, hence $\pi_n:F_n\twoheadrightarrow C_n$ is an $\cF$-precover.
 
Now assume $n\geq \sup\{j\mid \H_j(F)\not=0\}$. Let $\gamma_n:F_n\to F_n$ be such that $\pi_n\gamma_n=\pi_n$. Since, for $i\geq n+1$, each $C_i\cong\im(\del_i^F) \in \cC$, the map $\gamma_n$ extends to a map $\gamma_{\geq n}:F_{\geq n} \to F_{\geq n}$, thought of as a map of left resolutions of $C_n$ which lifts $1^{C_n}$. Moreover, $\gamma_{\geq n}\sim 1^{F_{\geq n}}$ by a standard argument, which we give in this case: Because $C_{n+2}\cong \im(\del_{n+2})\in \cC$, the natural map $\Hom_R(F_n,F_{n+1})\twoheadrightarrow \Hom_R(F_n,C_{n+1})$ is surjective; furthermore, because $\pi_n1^F_n-\pi_n\gamma_n=0$, we have $1^F_n-\gamma_n\in \Hom_R(F_n,C_{n+1})$, and so there is a map $\sigma_n:F_n\to F_{n+1}$ such that $1^F_n-\gamma_n=\del_{n+1}^F\sigma_n$. For $j> n$, we inductively assume that $\sigma_{j-1}$ and $\sigma_{j-2}$ have been constructed so that $1_{j-1}^F-\gamma_{j-1}=\sigma_{j-2}\del_{j-1}^F+\del_j^F\sigma_{j-1}$ (where we set $\sigma_{n-1}=0$), and claim that $(1^F_j-\gamma_j)-\sigma_{j-1}\del_j^F\in \Hom_R(F_j,C_{j+1})$.  This follows because 
\begin{align*}
\del_j^F((1^F_j-\gamma_j)-\sigma_{j-1}\del_j^F)&=\del_j^F1^F_j-\del_j^F\gamma_j-\del_j^F\sigma_{j-1}\del_j^F\\
&=\del_j^F1^F_j-\del_j^F\gamma_j-(\sigma_{j-2}\del_{j-1}^F+\del_j^F\sigma_{j-1})\del_j^F\\
&=\del_j^F1^F_j-\del_j^F\gamma_j-(1_{j-1}^F-\gamma_{j-1})\del_j^F\\
&=-\del_j^F\gamma_j+\gamma_{j-1}\del_j^F\\
&=0,
\end{align*}
since $\gamma$ is a chain map. As before, there exists a map $\sigma_{j}$ such that $(1^F_j-\gamma_j)-\sigma_{j-1}\del_j^F=\del_{j+1}^F\sigma_j$, hence $1^F_j-\gamma_j=\sigma_{j-1}\del_j^F+\del_{j+1}^F\sigma_j$. Thus $\gamma_{\geq n}\sim 1^{F_{\geq n}}$. Extend the map $\gamma_{\geq n}$ to a map $\gamma:F\to F$ by defining $\gamma_i=1^{F}_i$ for $i<n$ and set $\sigma_i=0$ for $i<n$. Inspection shows that $\gamma\sim 1^F$ via the homotopy $\sigma$. Minimality of $F$ implies that $\gamma$, and therefore $\gamma_n$, is an isomorphism. Thus $\pi_n:F_n\twoheadrightarrow C_n$ is an $\cF$-cover.

For an $R$-complex $C$ as in (2), the argument is dual. One first remarks that $C^i/\ker(\del^i_C)\cong \im(\del^i_C)\in \cF$, and so for any $E\in \cC$, we have $\Ext_R^1(C^i/\ker(\del^i_C),E)=0$ and thus maps $\ker(\del^i_C)\to E$ factor through $\ker(\del^i_C)\hookrightarrow C^i$, implying that the natural inclusions are $\cC$-preenvelopes. For $n\geq \sup\{j\mid H^j(C)\not=0\}$, one argues that any map $\gamma^n:C^n\to C^n$ induces a map $C\to C$ that is homotopic to $1_C$; minimality of $C$ yields the desired result.
\end{proof}

Combined with Theorem \ref{cov_env_min}, one consequence of Proposition \ref{converse_min} is that an acyclic complex $F$ of modules from $\cF$ with syzygies from $\cC$ is minimal if and only if $F_i\twoheadrightarrow \coker(\del_{i+1}^F)$ is an $\cF$-cover for each $i\in \Z$; an acyclic complex $C$ of modules from $\cC$ with syzygies from $\cF$ is minimal if and only if $\ker(\del^i_C)\hookrightarrow C^i$ is a $\cC$-envelope for each $i\in \Z$.

Another case of interest in the previous result is when $F$ (or $C$) is a left (or right) resolution of a module. Let $\cP$ be the class of projective modules, $\cI$ be the class of injective modules, and $\cM$ be the class of all $R$-modules, and recall \cite[Example 7.1.3]{EJ00} that $(\cP,\cM)$ and $(\cM,\cI)$ are cotorsion pairs. We immediately obtain that a projective resolution $P$ of a module is minimal if and only if $P_i\twoheadrightarrow \coker(\del_{i+1})$ is a $\cP$-cover for all $i\in \Z$; also, an injective resolution $I$ of a module is minimal if and only if $\ker(\del^i)\hookrightarrow I^i$ is an $\cI$-envelope for all $i\in \Z$.  
\begin{cor}\label{perfect_projective_covers}
Let $R$ be a ring. Then the following are equivalent:
\begin{enumerate}
\item Every $R$-module has a $\cP$-cover, i.e., $R$ is perfect;
\item Every $R$ module has a minimal projective resolution.
\end{enumerate}
\end{cor}
\begin{proof}
The implication $(1)\implies (2)$ follows from Theorem \ref{cov_env_min}.  For the other implication, since $(\cP,\cM)$ is a cotorsion pair, Proposition \ref{converse_min} then implies the canonical surjections in a minimal projective resolution are all $\cP$-covers, and (1) follows.
\end{proof}
The conditions of Corollary \ref{perfect_projective_covers} are also equivalent to every flat module being projective; additional equivalent conditions are contained in \cite[Theorem 1.2.13]{Xu96}.

\section{Minimal cotorsion flat resolutions and replacements}\label{section_minimal_CF}
As an application of the results in Section \ref{covers_envelopes_minimality}, the goal of this final section is to show the existence of cotorsion flat resolutions and replacements for modules. 

A {\em left (or right) cotorsion flat resolution} of an $R$-module $M$ is a complex $B$ of cotorsion flat $R$-modules with a quasi-isomorphism $B\xrightarrow{\simeq} M$ (or $M\xrightarrow{\simeq}B$), such that $B_{i}=0$ for $i<0$ (or for $i>0$). A {\em cotorsion flat replacement} of an $R$-module $M$ is a complex of cotorsion flat $R$-modules which is isomorphic to $M$ in $\D(R)$. We caution that left/right cotorsion flat resolutions only exist for certain modules, but cotorsion flat replacements exist for every module (see Theorem \ref{CFreplacement}).

\begin{example}\label{example}
Let $(R,\m)$ be a complete local ring.  Finitely generated projective $R$-modules are cotorsion flat; therefore, the minimal projective resolution of any finitely generated $R$-module is a minimal left cotorsion flat resolution.
\end{example}

Let $\Flat$, $\Cot$, and $\PurInj$ be the classes of flat, cotorsion, and pure-injective $R$-modules (see \cite[Definition 5.3.6]{EJ00} for a definition of pure-injective module). Recall our convention from \ref{covers_envelopes_resolutions} that a covering $\Flat$-resolution is built from $\Flat$-covers and an enveloping $\Cot$- (or $\PurInj$-)resolution is built from $\Cot$- (or $\PurInj$-)envelopes.  

\begin{thm}\label{CFreplacement}
Let $R$ be a commutative noetherian ring.
\begin{enumerate}
\item Every cotorsion $R$-module has a minimal left cotorsion flat resolution; the covering $\Flat$-resolution is such a resolution.
\item Every flat $R$-module has a minimal right cotorsion flat resolution; the enveloping $\Cot$- (or $\PurInj$-)resolution is such a resolution.
\item Every $R$-module is isomorphic in $\D(R)$ to a minimal semi-flat complex of cotorsion flat $R$-modules; more precisely, for any $R$-module $M$ there is a diagram of quasi-isomorphisms
$$B\xleftarrow{\simeq}F\xrightarrow{\simeq} M,$$
where $F$ is a minimal complex of flat $R$-modules and $B$ is a semi-flat minimal complex of cotorsion flat $R$-modules.
\end{enumerate}
\end{thm}

\begin{proof}
$(1)$: Let $L$ be a cotorsion $R$-module and $F\xrightarrow{\simeq} L$ its covering $\Flat$-resolution; such a resolution exists and is a left resolution (that is, the augmented sequence $\cdots \to F_1\to F_0\to L\to 0$ is exact) by \ref{covers_envelopes_resolutions}. As the $\Flat$-cover of a cotorsion module is cotorsion flat \cite[Corollary]{Eno84} and the kernel of an $\Flat$-cover is cotorsion \cite[Lemma 2.2]{Eno84}, we note that $F$ is a left cotorsion flat resolution; it is minimal by Theorem \ref{cov_env_min}.

$(2)$: Let $N$ be a flat $R$-module and $N\xrightarrow{\simeq} C$ its enveloping $\Cot$-resolution; such a resolution exists and is a right resolution (i.e., the augmented resolution is exact) by \ref{covers_envelopes_resolutions}. The $\Cot$-envelope of a flat module is cotorsion flat and its cokernel is flat \cite[Theorem 3.4.2]{Xu96}; see also \cite{GJ81} and \cite[Lemma 1.1 and discussion following]{Eno87}. Thus $C$ is a right cotorsion flat resolution, which is minimal by Theorem \ref{cov_env_min}. As the $\Cot$-envelope and the $\PurInj$-envelope of a flat module are isomorphic \cite[Remark 3.4.9]{Xu96}, we see that $C$ is isomorphic to the enveloping $\PurInj$-resolution.

$(3)$: Let $F\xrightarrow{\simeq}M$ be the covering $\Flat$-resolution of $M$ and let $F_0\xrightarrow{\simeq}C$ be the enveloping $\Cot$-resolution of $F_0$.  Stitch these resolutions together as follows:
$$B_i=\begin{cases}F_i, & i>0;\\ C^{-i}, & i\leq 0;\end{cases}$$
with differential
$$\del_i^B=\begin{cases}\del_i^F, & i\geq 2;\\ \iota\circ\del_1^F, & i=1;\\ \del_C^{-i}, & i\leq 0;\end{cases}$$
where $\iota:F_0\hookrightarrow{C^0}$ is the augmentation map. As $F_0$ is flat, $C^{-i}$ is cotorsion flat for $i\leq 0$ by (2); since $\ker(F_0\to M)$ is cotorsion by \cite[Lemma 2.2]{Eno84}, it follows that $F_i$ is cotorsion flat for $i\geq 1$ by (1). Thus $B$ is a complex of cotorsion flat modules such that $B\xleftarrow{\simeq}F\xrightarrow{\simeq} M$ is a diagram of quasi-isomorphisms.

To verify that $B$ is minimal, Theorem \ref{minimal_CF_complexes} tells us it is enough to ensure that $B$ has no subcomplex of the form $\cdots \to 0 \to \widehat{R_\p}^\p\xrightarrow{\cong} \widehat{R_\p}^\p\to 0\to \cdots$ which is degreewise a summand of $B$. If such a forbidden subcomplex of $B$ were to exist, it would follow that one would have to exist as a subcomplex of either $F$ or $C$. This is clear in homological degrees at least 1 or at most 0, and for such a two term subcomplex of $B$ concentrated in homological degrees 1 and 0, one uses the injection $\iota$ to show that it must also induce such a subcomplex of $F$.  However, the complexes $F$ and $C$ are both minimal by Theorem \ref{cov_env_min}, and so no such forbidden subcomplex can exist in $B$. Thus $B$ is a minimal complex.

Finally, we show $B$ is semi-flat: There is a short exact sequence of $R$-complexes $0\to F\to B\to B/F\to 0$. The $R$-complex $B/F$ is isomorphic to $0\to C^0/F_0\to C^1\to C^2\to \cdots$, which is semi-flat because it is an acyclic complex of flat modules having flat syzygies by \cite[Lemma 2.1.2]{Xu96}; the complex $F$ is semi-flat because it is a bounded on the right complex of flat $R$-modules. Therefore, the short exact sequence implies that $B$ is also semi-flat.
\end{proof}

\noindent
{\bf Acknowledgements:} Much of the present work comes from a portion of the author's dissertation at the University of Nebraska-Lincoln.  The author is greatly indebted to his advisor, Mark Walker, whose advice and support has been invaluable. Many thanks are also owed to Lars Winther Christensen, Douglas Dailey, Thomas Marley, and the anonymous referee for helpful suggestions.

\bibliographystyle{plain}

\end{document}